\documentclass[10pt,a4paper,reqno]{amsart}  

\usepackage{mathtools}
\usepackage{url}
\usepackage[breaklinks]{hyperref}

\usepackage[headings]{fullpage}
\usepackage{datetime}  
\usepackage{bm}
\usepackage{soul}
\usepackage{cancel}
\usepackage[margin=1in]{geometry}
\usepackage{amsfonts}
\usepackage{amsmath}
\usepackage{amsthm}
\usepackage{comment}
\usepackage{epsfig}
\usepackage{graphics}
\usepackage{mathrsfs}
\usepackage{paralist}
\usepackage{psfrag}
\usepackage{xcolor}
\usepackage{tikz}
\usepackage{centernot}
\usepackage{faktor}
\usepackage{xfrac}
\usepackage{amssymb}
\usepackage{stackrel}
\usepackage{listings}
\usepackage{verbatim}
\usepackage{multirow}
\usepackage{array}
\usepackage{float}
 
\newcommand{\z}{\mathbb}

\newcommand{\N}{\mathbb{N}}

\newcommand{\Q}{\mathbb{Q}}
\newcommand{\Z}{\mathbb{Z}}
\newcommand{\al}{\alpha}

\newtheoremstyle{sltheorems}% name
{10pt}%      Space above
{6pt}%      Space below
{\slshape}%         Body font
{}%         Indent amount (empty = no indent, \parindent = para indent)
{\bfseries}% Thm head font
{.}%        Punctuation after thm head
{.5em}%     Space after thm head: " " = normal interword space;
{\thmname{#1}\thmnumber{ #2}\thmnote{ (#3)}}
\theoremstyle{sltheorems} 
\newtheorem{Theorem}{Theorem}[section]
\newtheorem{Lemma}[Theorem]{Lemma} 
\newtheorem{Corollary}[Theorem]{Corollary} 
\newtheorem{Definition}[Theorem]{Definition} 
\newtheorem{Proposition}[Theorem]{Proposition}

\newcolumntype{C}[1]{>{\centering\let\newline\\\arraybackslash\hspace{0pt}}m{#1}}

\allowdisplaybreaks

\author{Leonardo Carofiglio, Luigi De Filpo, Alessandro Gambini}

\title{$p$-adic valuation of harmonic sums and their connections with Wolstenholme primes}

\begin{document}

\begin{abstract}
We explore a conjecture posed by Eswarathasan and Levine on the distribution of $p$-adic valuations of harmonic numbers $H(n)=1+1/2+\cdots+1/n$ that states that the set $J_p$ of  the positive integers $n$ such that $p$ divides the numerator of $H(n)$ is finite. We proved two results, using a modular-arithmetic approach, one for non-Wolstenholme primes and the other for Wolstenholme primes, on an anomalous asymptotic behaviour of the $p$-adic valuation of $H(p^mn)$ when  the $p$-adic valuation of $H(n)$ equals exactly 3.
\end{abstract}

\maketitle

\small

\normalsize
\section{Introduction and general setting}

\maketitle

\noindent The $n$-th harmonic number is defined as the partial sum of the well-known harmonic series as follows: 
\[
H(n):=1+\frac12+\frac13+\cdots+\frac1n.
\]
It has been known since the nineteenth century that by the Wolstenholme Theorem \cite{Wolstenholme},  for any prime $p\ge5$ the numerator of $H(p-1)$ is multiple of $p^2$. Many researchers have investigated the arithmetic properties of harmonic numbers and related problems; in the last century, Bleicher and Erd\"os have studied the so-called harmonic subsums \cite{BleicherE1975} while random harmonic sums and harmonic series have been investigated in probabilistic terms (see, for example, \cite{Bettin, GTZ,  Schmuland, Worley}).

In 1991, Eswarathasan and Levine \cite{EL} introduced the set $J_p$ of the positive integers $n$ whereby $H(n)$ is a multiple of $p$ and conjectured that $J_p$ is finite for all primes $p$. Eswarathasan and Levine also introduced an algorithm to count the elements of $J_p$ in the event that $J_p$ is finite, which was improved by Boyd \cite{Boyd} in 1994. Boyd determined $J_p$ for all $p\le 547$, except 83, 127 and 397. It is easy to show that $J_2=\emptyset$ and $J_3=\{2,7,22\}$, while for all $p\geq 5$, 
\[
\{p-1,\,p(p-1),\,p^2-1\}\subseteq J_p.
\]
The primes for which $J_p$ contains only these three elements are called ``harmonic'': for instance $p=5$ is harmonic as $J_5=\{4,20,24\}$ (see \cite{EL} for further details). The case $p=3$ was also treated individually in the paper by Kamano \cite{kamano}.

\noindent Sanna \cite{sanna} recently gave an upper bound for the number of elements of $J_p$ showing that
\[
\#J_p(x)\le 129 p^{\frac{2}{3}}x^{0.765}
\] 
where $J_p(x)=J_p\cap[1,x]$. This upper bound was improved by Chen and Wu \cite{wu-chen}:
\[
\#J_p(x)\le 3x^{\frac{2}{3}+\frac{1}{25\log p}}.
\]
They also tackled the alternating harmonic sum \cite{wu-chen2}. Recently Leonetti and Sanna also studied the $p$-adic valuation of $H(n, k)$, for $n\geq k$, a kind of generalized Harmonic numbers defined as:
\[
H(n,k):=\sum_{1\le i_1< \cdots < i_k\le n}\frac{1}{i_1\cdots i_k},
\]
which are closely linked to the Stirling numbers of the first kind.

An interesting connection between harmonic numbers and Bernoulli numbers was discovered by Boyd \cite[\S 4]{Boyd}. This connection concerns Wolstenholme primes that are defined as the primes $p$ that divide the numerator of the Bernoulli numbers $B_{p-3}$. The only two known Wolstenholme primes are 16843 and 2124679, but it is conjectured that infinitely many such primes exist. The latest search for Wolstenholme primes in 2007 found that those were the only two up to $10^9$ (see \cite{mcintosh}).
In this paper, we explored this connection by identifying different behaviours of $J_p$ according to whether $p$ is or is not a Wolstenholme prime.

We studied the $p$-adic valuation of the sequence of harmonic numbers $H(p^mn)$ for fixed $(n,p)$ where $p\nmid n$, and we identified that it follows one out of three patterns. These patterns differ according to whether $p$ is a Wolstenholme prime or not. Interesting phenomena occur in pairs $(n,p)$ such that $\nu_p(H(n))=3$, whose only known occurrences up to today were $n=848,9338,10583,3546471722268916272$ for $p=11$, one $n\geq 10^5$ for $p=83$, which were found by Boyd \cite{Boyd}, $(n=16842,\,p=16843$) and $(n=2124678,\,p=2124679)$, which are Wolstenholme primes and satisfy $\nu_{16843}(H(16842))=3$ and $\nu_{2124678}(H(2124679))\geq 3$. Finally, Boyd \cite{Boyd} conjectured that there are no pairs $(n,p)$ such that $\nu_p(H(n))\geq 4$.

\section{Results}

In \cite{Boyd} and \cite{EL} the following Lemma has been proved:
\begin{Lemma}\label{Boyd_3.1}
For any  prime $p\geq 5$, if $\nu_p(H(n))\leq 2$ then 
\[
\nu_p(H(pn))=\nu_p(H(n))-1.
\]
\end{Lemma}

Boyd \cite{Boyd} gives a stronger version of this lemma, which we have postponed to Proposition \ref{comp_somma}. At this introductory stage, in fact, a weaker version is sufficient to explains what we will call the ``descent phenomenon'' observed in Table \ref{tabella} of the $p$-adic valuations of harmonic numbers. Starting from that point, we can see that the behaviour of harmonic numbers when $\nu_p(H(n))\leq 2$ in terms of $p$-adic valuation is the following: $\nu_p(H(p^mn))=\nu_p(H(n))-m$ for all $m\in\N$; here we give a fragment for the table of $p=5$.

\begin{table}[h]
\caption{Table for $p=5$: this shows the $5$-adic valuation of $H(5m+k)$, where $m$ is the column index and $k$ is the row index.}\label{tabella}
\centering
\begin{tabular}{ | C{1cm} || C{1cm} | C{1cm} | C{1cm} | C{1cm} | C{1cm} | }
 \hline
 \multicolumn{6}{|c|}{$\nu_5(H(5m+k))$} \\
 \hline
 	&	$k=0$	& $k=1$	& $k=2$	& $k=3$	& $k=4$	\\
 	\hline
 	\hline
$m=0$	& 	$\infty$	&0	& 0	& 0	& 2	\\
1		&-1	&-1	&-1	&-1	&-1	\\
2		&-1	&-1	&-1	&-1	&-1	\\
3		&-1	&-1	&-1	&-1	&-1	\\
4		& 1	& 0	& 0	& 0	& 1	\\
5		&-2	&-2	&-2	&-2	&-2	\\
6		&-2	&-2	&-2	&-2	&-2	\\
$\vdots$ & $\vdots$ & $\vdots$ & $\vdots$ & $\vdots$ & $\vdots$\\
19		&-2	&-2	&-2	&-2	&-2	\\
20		& 0	& 0	& 0	& 0	& 0	\\
21		&-1	&-1	&-1	&-1	&-1	\\
22		&-1	&-1	&-1	&-1	&-1	\\
23		&-1	&-1	&-1	&-1	&-1	\\
24		& 0	& 0	& 0	& 0	& 0	\\
25		&-3	&-3	&-3	&-3	&-3	\\
$\vdots$ & $\vdots$ & $\vdots$ & $\vdots$ & $\vdots$ & $\vdots$\\
124		&-1	&-1	&-1	&-1	&-1	\\
125		&-4	&-4	&-4	&-4	&-4	\\
126		&-4	&-4	&-4	&-4	&-4	\\
$\vdots$ & $\vdots$ & $\vdots$ & $\vdots$ & $\vdots$ & $\vdots$\\
\hline
\end{tabular}
\end{table}

However, Lemma \ref{Boyd_3.1} cannot be applied to the case $\nu_p(H(n))\geq 3$. We provide formulas that can apply to this case as well. 
\begin{Definition}
The Bernoulli numbers $B_n$ are defined using the generating function 
\[
\frac{t}{1-e^t}=\sum\limits_{k=0}^{\infty}B_k \frac{t^k}{k!}.
\]
\end{Definition}

\begin{Definition}
A prime $p$ is a Wolstenholme prime if $p\mid B_{p-3}$.
\end{Definition}

\begin{Proposition}
$p$ is a Wolstenholme prime if and only if $p^3\mid H(p-1)$.
\end{Proposition}

\noindent This result explains the fact that Wolstenholme primes give occurrences of $\nu_p(H(n))= 3$ for $n=p-1$ as in the introduction. It is important to report this result (which comes from a paper submitted by Gardiner \cite{gardiner} in 1988), as it shows also other equivalent definitions of Wohstenholme primes, that, however, we did not use in this article.

Below we present the main results of our research:

\begin{Theorem}\label{non-W-primes}
Let $p\geq 5$ be a non-Wolstenholme prime and $n\in\N$ such that $p\nmid n$ and $\nu_p(H(n))\geq 3$. Then
\begin{enumerate}
\item If $\nu_p(H(n))\geq 4$, then $\nu_p(H(pn))=2$ and then we perceive the ``descent phenomenon'',
\item If $\nu_p(H(n))=3$ and there does not exists $m\in\N$ such that $\nu_p(H(p^mn)) \leq 2$, then $\nu_p(H(p^mn))=2m+3$ for all $m\in\N$,
\item If $\nu_p(H(n))=3$ and there exists $m\in\N$ such that $\nu_p(H(p^mn)) \leq 2$, then there exists $M\in\N$ such that $\nu_p(H(p^Mn))=0$ and 
\[
\nu_p(H(p^mn))=
\begin{cases}
\begin{aligned}
&2m+3&\quad\text{if }m\leq \frac{M}{3}+1\\
&M-m&\quad\text{if }m> \frac{M}{3}+1
\end{aligned}
\end{cases}
\]
for all $m\in\N$.
\end{enumerate}
\end{Theorem}

\begin{Theorem}\label{W-primes}
Let $p\geq 5$ be a Wolstenholme prime, and $n\in\N$ such that $p\nmid n$ and $\nu_p(H(n))\geq 3$. Then
\begin{enumerate}
\item If $\nu_p(H(n))=3$, then $\nu_p(H(pn))=2$ and then we perceive the ``descent phenomenon'',
\item If $\nu_p(H(n))\geq 4$ and there does not exists $m\in\N$ such that $\nu_p(H(p^mn)) \leq 3$, then $\nu_p(H(p^mn))\geq 2m+4$ for all $m\in\N$,
\item If $\nu_p(H(n))\geq 4$ and there exists $m\in\N$ such that $\nu_p(H(p^mn)) \leq 3$, then there exists $M\in\N$ such that $\nu_p(H(p^Mn))=0$ and
\[
\begin{cases}
\begin{aligned}
&\nu_p(H(p^mn))\geq 2m+4&\quad\text{if }m\leq \frac{M+1}{3}+1\\
&\nu_p(H(p^mn))=M-m&\quad\text{if }m> \frac{M+1}{3}+1
\end{aligned}
\end{cases}
\]
for all $m\in\N$.
\end{enumerate}
\end{Theorem}

\section{Setting the problem}

In this section we will use the Big-O notation,
\[
\nu_p(x-y)\geq k\quad \Longleftrightarrow\quad x-y=O(p^k).
\]

\begin{Proposition}\label{equiv_p-adica}
Let $a,b\in\Z$, $(b,p)=1$. Then
\[
\frac{a}{b}=
\frac{a}{b+hp}+O(p)
\]
for all $h\in\Z$.
\end{Proposition}
\begin{proof}
We have
\[
\nu_p\left(\frac{a}{b}-\frac{a}{b+hp}\right)=\nu_p\left(\frac{ahp}{b(b+hp)}\right)\geq 1,
\]
\end{proof}

\begin{Proposition}\label{equiv_congr-val}
Let $q\in\Q$. Then $\nu_p(q)=k<1$ if and only if there exist $a,b\in\Z$, $b\neq 0$ with $(a,p)=1$ and $(b,p)=1$ such that 
\[
q=
p^k\frac{a}{b}+O(p).
\] 
\end{Proposition}
\begin{proof}
\begin{itemize}
    \item[($\Longrightarrow$)] Straightforward.
    \item[($\Longleftarrow$)] We have
\[
\nu_p\left(q-p^k\frac{a}{b}\right)\geq 1,
\]
then
\[
\nu_p(q)=\min\left\{\nu_p\left(p^k\frac{a}{b}\right),\nu_p\left(q-p^k\frac{a}{b}\right)\right\}=k.
\]
\end{itemize}

\vspace{-.5cm}

\end{proof}

Now, in order to study the behaviour of $J(p)$, we must consider a very important property: 

\begin{Lemma}\label{lemma:discesa}
If $n\in J(p)$, $\nu_p(H(pn))\geq 0$,  we have
\[
H(pn+k)=\frac{H(n)}{p}+H(k)+O(p)
\]
for each $k=1,\dotsc p-1$.
\end{Lemma}
\begin{proof}
Using Lemma \ref{Boyd_3.1} and Proposition \ref{equiv_p-adica} we get
\[
H(pn+k)
=
H(pn)+\sum\limits_{i=1}^k\frac{1}{pn+i}
=
\frac{H(n)}{p}+\sum\limits_{i=1}^k\frac{1}{i}+O(p)
=
\frac{H(n)}{p}+H(k)+O(p),
\]
\end{proof}

From Lemma \ref{lemma:discesa} follows immediately the following:
\begin{Corollary}
If $\nu_p(H(pn))\geq 1$,  $H(pn+k)= H(k)+O(p)$.
\end{Corollary}
It means that the set of indices $k=1,\dotsc p-1$ such that $\nu_p(H(pn+k))=0$ coincides with the set of indices $k=1,\dotsc p-1$ as before such that $\nu_p(H(k))=0$. This explains the repetition with the same patterns of zeros observed in Table \ref{tabella}, however we cannot yet predict the actual $p$-adic valuation, only whether it is null or not.
If, on the other hand, $\nu_p(H(pn))=0$, then given that 
\[
H(pn+k)
=
\frac{H(n)}{p}+H(k)
\]
for each $k=1,\dotsc p-1$, the study of these $k$ with $H(pn+k)= 0$ is reduced to the study of the $p$-adic valuation in modulus $p$ of the initial terms of a shifted sum. We now focus on harmonic numbers of the type $H(p^mn)$.

\subsection{Formulas for $H(p^mn)$}

We recall a stronger version of Lemma \ref{Boyd_3.1}: 
\begin{Proposition}[see \cite{Boyd}, Lemma 3.1]\label{comp_somma}
If $p\geq 5$ is prime and $n\ge1$, we have 
\[
H(pn)= \frac{H(n)}{p}+O(p^2).
\]
\end{Proposition}

\noindent We could conclude that, if $\nu_p(H(n))\leq 2$, then 
\[
\nu_p(H(pn))=\nu_p(H(n))-1.
\]

\begin{Lemma}\label{casobase_m=2}
If $p\geq 5$ is prime and $n\ge1$, we have
\[
H(p^2n)=\frac{H(n)}{p^2}+O(p),
\]
\end{Lemma}

\begin{proof}
Proposition \ref{comp_somma} implies that there exists $\al\in\Q$ with $\nu_p(\al)\geq 0$ such that 
\[
H(pn)=\frac{H(n)}{p} +O(p^2),
\]
then
\[
H(p^2n)
=
\frac{H(pn)}{p}+O(p)
=
\frac{H(n)}{p^2}+O(p).
\]
\end{proof}

\begin{Lemma}
Let $a,b\in\Z$, then $(b,p)=1$ and $k\geq 0$ an integer.
Then
\[
\frac{a}{p^kb}=\frac{ab^{-1}}{p^k}+O(p),
\]
where $b^{-1}$ denotes the inverse of $b$ modulo $p^{k+1}$.
\end{Lemma}
\begin{proof}
We have
\[
\nu_p\left(\frac{a}{p^kb}-\frac{ab^{-1}}{p^k}\right)=\nu_p\left(\frac{a(1-bb^{-1})}{p^kb}\right)\geq 1.
\]
\end{proof}

\begin{Lemma}
If $a,b\in\Z$ and $a=b+O(p^{k+1})$, 
\[
\frac{a}{p^k}= \frac{b}{p^k}+O(p).
\]
\end{Lemma}
\begin{proof}
We have
\[
\nu_p\left(\frac{a}{p^k}-\frac{b}{p^k}\right)=\nu_p\left(\frac{a-b}{p^k}\right)=\nu_p(a-b)-k\geq 1.
\]
\end{proof}

\noindent We now recall the famous Clausen-Van Staudt Theorem:
\begin{Theorem}[Clausen-Van Staudt]
Let $n\in\N$ and  $B_i$ be the $i$-th Bernoulli number, then
\[
B_{2n}+\!\!\sum\limits_{\substack{{p\;\mathrm{ prime}}\\{(p-1)\mid 2n}}}\frac{1}{p}
\]
is an integer.
\end{Theorem}
See \cite{Carlitz} for a proof of the theorem. 

\begin{Corollary}\label{CorCVS}
Let $n\in\N$ and  $B_i$ be the $i$-th Bernoulli number, then
\begin{itemize}
    \item [i)] $\nu_p(B_{2n})\geq -1,$
    \item [ii)] $\nu_p(B_{2n})=-1\quad \iff \quad p-1\mid 2n.$
\end{itemize}
\end{Corollary}

The previous Propositions and Lemmas are used to investigate $H(p^mn)$ more in detail:
\begin{Theorem}
\label{Formula1}Let $n,m$ be non-negative integers and $p\geq 5$ be a prime number, then the following formula holds
\begin{equation}\label{m_power:first}
H(p^mn)
=
\frac{H(n)}{p^m}
+\sum\limits_{h=2}^{m-1}
\sum\limits_{k=1}^{h}
\frac{B_{p^h(p-1)-2k}}{2kp^{h}}\binom{p^{h}(p-1)-1}{2k-1}p^{2k(m-h)}n^{2k}+O(p),
\end{equation}
where $B_i$ is the $i$-th Bernoulli number.
\end{Theorem}

\begin{proof}

In order to prove the statement:

\noindent Let us define
\[
a_m(n)
:=
\sum\limits_{h=2}^{m-1}
\sum\limits_{k=1}^{h}
\frac{B_{p^h(p-1)-2k}}{2kp^{h}}
\binom{p^{h}(p-1)-1}{2k-1}
p^{2k(m-h)}n^{2k}. 
\]
Hence the statement becomes:
\[
H(p^mn)
=
\dfrac{H(n)}{p^m}+a_m(n)+O(p).
\]
The proof follows by induction on $m$: our aim is to find functions $\chi_m(n)$ such that the following congruences hold:
\[
H(p^{m+1}n)
=
\dfrac{H(pn)}{p^m}+a_m(n)+O(p)
=
 \dfrac{H(n)}{p^{m+1}}+\chi_m(n)+a_m(n)+O(p)
 =
 \dfrac{H(n)}{p^{m+1}}+a_{m+1}(n)+O(p).
\]
\begin{itemize}
    \item [\textit{Step 1.}] Find $\chi_m(n)$ such that
    \[
    \dfrac{H(pn)}{p^m}
    =
    \dfrac{H(n)}{p^{m+1}}+\chi_m(n)+O(p).
    \]
    \item [\textit{Step 2.}] Prove that, for the $\chi_m(n)$ already found, it holds
    \[
    \chi_m(pn)+a_m(n)
    =
    a_{m+1}(n)+O(p).
    \]
\end{itemize}

\noindent\textit{Proof of step 1.} By Proposition \ref{comp_somma} and noting that \[\displaystyle H(pn)=\sum\limits_{k=1}^n\frac{1}{pk}+\sum\limits_{\ell=0}^{n-1}\sum\limits_{k=1}^{p-1}\frac{1}{\ell p+k},\] we obtain
\[
\begin{aligned}
\dfrac{H(pn)}{p^m}
&=
\dfrac{1}{p^m}\left(\sum\limits_{k=1}^n\frac{1}{pk}+\sum\limits_{\ell=0}^{n-1}\sum\limits_{k=1}^{p-1}\frac{1}{\ell p+k}\right)+O(p)
\\
&=
\dfrac{1}{p^m}\left(\dfrac{H(n)}{p}+\sum\limits_{\ell=0}^{n-1}\sum\limits_{k=1}^{p-1}\frac{1}{\ell p+k}\right)+O(p)
\\
&=
\dfrac{H(n)}{p^{m+1}}+\dfrac{1}{p^m}\sum\limits_{\ell=0}^{n-1}\sum\limits_{k=1}^{p-1}\frac{1}{\ell p+k}+O(p).
\end{aligned}
\]
Hence, we found
\[
\chi_m(n)=\dfrac{1}{p^m}\sum\limits_{\ell=0}^{n-1}\sum\limits_{k=1}^{p-1}\frac{1}{\ell p+k}.
\]
We focus on the inner sum of the RHS:
\begin{align}\label{main:telescopic}
\frac{1}{p^m}\sum_{k=1}^{p-1}\frac{1}{\ell p+k}
&= 
\frac{1}{p^m}\sum_{k=1}^{p-1}(\ell p+k)^{p^m(p-1)-1}+O(p)
\notag
\\
&=
\frac{1}{p^m}
\left(\sum_{k=1}^{\ell p+p-1}k^{p^m(p-1)-1}
-
\sum_{k=1}^{\ell p}k^{p^m(p-1)-1}\right)+O(p)
\notag
\\
&=
\frac{1}{p^m}
\left(\sum_{k=1}^{(\ell+1)p}k^{p^m(p-1)-1}
-
\sum_{k=1}^{\ell p}k^{p^m(p-1)-1}\right)+O(p).
\end{align}
We have taken advantage of the fact that, if $y$ has an inverse modulo $n$, $y^{-1}\equiv y^{\varphi(n)-1}\text{(mod }n\text{)}$ by Euler's Theorem, where $\varphi$ denotes Euler's totient function. This explains the choice of the exponent $p^m(p-1)=\varphi(p^{m+1})$.

\noindent Now the recurring expression
\[
\frac{1}{p^m}\sum_{k=1}^{\ell p}k^{p^m(p-1)-1}
\]
can be simplified using Bernoulli's formula for sums of powers (Faulhaber's formula, see \cite[\S 4]{Conway} pag. 106):
\begin{align}\label{faulhaber}
\frac{1}{p^m}\sum\limits_{\substack{k=1}}^{\ell p}k^{p^m(p-1)-1}
=
\frac{1}{p^{2m}(p-1)}\!\!\!\!\!\!\!\!
\sum_{k=0}^{p^m(p-1)-1}
\!\!\!\!\!\!\!(-1)^k B_k\binom{p^m(p-1)}{k}(\ell p)^{p^m(p-1)-k}.
\end{align}
In terms of congruence modulo $p$, the terms with $0\leq k< p^m(p-1)-2m$ vanish since they have positive $p$-adic valuation: by Corollary \ref{CorCVS} we find that $\nu_p(B_k)\geq -1$. Moreover, the $p$-adic valuation of a binomial coefficient is non-negative and by removing $B_1$ we do not need the factor $(-1)^k$ because $B_k=0$ for all $k\ge3$ odd. Then, the RHS of \eqref{faulhaber} is equal to
\[
\frac{1}{p^{2m}(p-1)}
\sum_{k=1}^{2m}B_{p^m(p-1)-k}
\binom{p^m(p-1)}{k}(\ell p)^{k}.
\]
Resuming from \eqref{main:telescopic}, we arrive at
\begin{align*}
\frac{1}{p^m}\sum_{k=1}^{p-1}\frac{1}{\ell p+k}
&=
\frac{1}{p^m}
\left(\sum\limits_{\substack{k=1}}^{(\ell +1)p}k^{p^m(p-1)-1}
-
\sum\limits_{\substack{k=1}}^{\ell p}k^{p^m(p-1)-1}\right)+O(p)
\\
&=
\frac{1}{p^{2m}(p-1)}
\sum_{k=1}^{2m}B_{p^m(p-1)-k}
\binom{p^m(p-1)}{k}
\left(\left((\ell +1)p\right)^k-(\ell p)^k\right)+O(p).
\end{align*}

\noindent Now we have a candidate $\chi_m(n)$, and by rearranging terms, developing the telescopic sum and highlighting again that $B_k=0$ for all $k\geq 3$ odd, we arrive at the following

\begin{align*}
\chi_m(n)
&=
\frac{1}{p^{2m}(p-1)}
\sum_{\ell=0}^{n-1}
\sum_{k=1}^{2m}B_{p^m(p-1)-k}
\binom{p^m(p-1)}{k}
\left(\left((\ell +1)p\right)^k-(\ell p)^k\right)+O(p)
\\
&=
\frac{1}{p^{2m}(p-1)}
\sum_{k=1}^{2m}B_{p^m(p-1)-k}
\binom{p^m(p-1)}{k}(np)^k+O(p)
\\
&=
\frac{1}{p^{m}}
\sum_{k=1}^{2m}\frac{B_{p^m(p-1)-k}}{k}
\binom{p^m(p-1)-1}{k-1}(np)^k+O(p)
\\
&=
\sum_{k=1}^{m}
\frac{B_{p^m(p-1)-2k}}{2kp^m}
\binom{p^m(p-1)-1}{2k-1}(np)^{2k}+O(p).
\end{align*}

\noindent\textit{Proof of step 2.} In order to conclude the global proof we need to show that it holds
\[
a_{m+1}(n)
=
 \chi_m(pn)+a_m(n)+O(p),
\]
The identity is true for $m=3$ and, given the inductive hypothesis, the identity is trivial since $\chi_m(n)$ just found is exactly the $m$-th term of the sum of $a_m(n)$.
\end{proof}

\section{Proof of Theorem \ref{non-W-primes} and \ref{W-primes}}

\noindent We recall here this famous result from Kummer:
\begin{Theorem}[Kummer's Congruence]
\label{KummersConguence}
If $a,b\in\N$ are even, not divisible by $p-1$ and $a= b+O(p-1)$, then
\[
\frac{B_a}{a}
=
\frac{B_b}{b}+O(p).
\]
\end{Theorem}

\noindent Now we are able to prove Theorem \ref{non-W-primes}.

\subsection{Proof of Theorem \ref{non-W-primes}}

\begin{proof}[Proof of Theorem \ref{non-W-primes}]
\noindent
\begin{enumerate}
\item If we prove that $H(p^3n)= O(p)$ it follows from Lemma \ref{casobase_m=2} that $\frac{H(pn)}{p^2}= O(p)$ and this means that $\nu_p(H(pn))\ge2$. We consider then Theorem \ref{Formula1} for $m=3$:
\[
H(p^3n)
=
 \frac{H(n)}{p^3}-B_{p^2(p-1)-2}\frac{n^2}{2}+O(p)
 =
 \frac{H(n)}{p^3}-\frac{B_{p-3}}{3}n^2+O(p)
 =
  -\frac{B_{p-3}}{3}n^2+O(p),
\]
as by hypothesis, $\nu_p(H(n))\ge 4$. The RHS has valuation zero if $p$ is non-Wolstenholme: in fact, using Theorem \ref{KummersConguence} to justify the Bernoulli number congruence,
\[
-\frac{B_{p^2(p-1)-2}}{2}
=
\frac{B_{p^2(p-1)-2}}{p^2(p-1)-2}+O(p)
=
\frac{B_{p-3}}{p-3}+O(p)
=
 -\frac{B_{p-3}}{3}+O(p).
\]

\item By hypothesis, for all $m\in\z{N}$, $H(p^mn)=O(p)$, we then evaluate formula \eqref{m_power:first} in $p^jn$: $\forall j\in\N$, where $p\nmid n$,
\begin{align*}
H(p^{m+j}n)
&=
O(p)\quad\mbox{and}
\\
H(p^{m+j}n)
&= 
\frac{H(p^jn)}{p^m}
+\sum\limits_{h=2}^{m-1}
\sum\limits_{k=1}^{h}
\frac{B_{p^h(p-1)-2k}}{2kp^{h}}\binom{p^{h}(p-1)-1}{2k-1}p^{2k(m-h+j)}n^{2k}+O(p).
\end{align*}

\noindent We now aim to prove that the term with the lowest $p$-adic valuation in the sum is obtained for $h=m-1,k=1$. This term has valuation

\begin{align*}
v_{\min}=&\nu_p\left(\frac{B_{p^{m-1}(p-1)-2}}{2p^{m-1}}\binom{p^{m-1}(p-1)-1}{1}p^{2(j+1)}n^{2}\right)=\\
&=\nu_p\left(B_{p^{m-1}(p-1)-2}\right)+2j+3-m=\\
&=\nu_p\left(B_{p-3}\right)+2j+3-m=2j+3-m.
\end{align*}

\noindent Once again we used Theorem \ref{KummersConguence} and Proposition \ref{equiv_congr-val} to justify the Bernoulli number congruence
\[
-\frac{B_{p^{m-1}(p-1)-2}}{2}
=
\frac{B_{p^{m-1}(p-1)-2}}{p^{m-1}(p-1)-2}+O(p)
=
\frac{B_{p-3}}{p-3}+O(p)
=
-\frac{B_{p-3}}{3}+O(p),
\]
from which follows the equivalence of valuations if $p$ is non-Wolstenholme.
Now we discuss the $p$-adic valuation term by term in the sum. Each summand may be represented as a pair $(h,k)$ as coordinates in the two sums formula with $2\leq h\leq m-1$ and $1\leq k\leq h$. Therefore each term falls into one of these three groups:
\begin{itemize}
    \item [$i)$] The term corresponding to $(h,k)=(m-1,1)$, discussed above,
    \item [$ii)$] Terms of the forms $(h, k)$ such that $(p-1)\mid 2k$ (except the summand in the $(i)$ group),
    \item [$iii)$] The remaining terms.\\
\end{itemize}

\noindent For the general summand in \eqref{m_power:first} we have:

\begin{spreadlines}{1.5ex}
\begin{align}\label{inequality-non-W}
\nu_p
&\left(\frac{B_{p^h(p-1)-2k}}{2kp^{h}}
\binom{p^{h}(p-1)-1}{2k-1}p^{2k(m-h+j)}n^{2k}\right)=\notag
\\
&=\nu_p\left(B_{p^h(p-1)-2k}\right)+\nu_p\binom{p^{h}(p-1)-1}{2k-1}+2k(m-h+j)-h-\nu_p(k) \notag
\\
&\geq\nu_p\left(B_{p^h(p-1)-2k}\right)+2k(m-h+j)-h-\nu_p(k) \notag
\\
&\geq\nu_p\left(B_{p^h(p-1)-2k}\right)+2k(m-h+j)-h+(k-\nu_p(k)-1)-k+1 \notag
\\
&\geq\nu_p\left(B_{p^h(p-1)-2k}\right)+2k(m-h+j)-h-k+1 \notag
\\
&\geq\nu_p\left(B_{p^h(p-1)-2k}\right)+2k\Big(m-h+j-\frac{1}{2}\Big)-h+1, 
\end{align}
\end{spreadlines}
as $k>1+\nu_p(k)$ for $p\geq 5$, and $k\geq 2$. Now, we use Corollary \ref{CorCVS}: $\nu_p\left(B_{p^h(p-1)-2k}\right)\geq -1$ and $\nu_p\left(B_{p^h(p-1)-2k}\right)=-1$ if and only if $(p-1)\mid 2k$. 

\noindent If the terms belong to group ($ii$), $(p-1)\mid 2k$, then \eqref{inequality-non-W} is greater than
\begin{align*}
& 2k\Big(m-h+j-\frac{1}{2}\Big)-h \\
\geq &(p-1)(m-h+j-\frac{1}{2})-m+1\\
\geq &(p-1)(1+j-\frac{1}{2})-m+1>v_{\min}.
\end{align*}

\noindent Otherwise, if the terms belong to group ($iii$), the inequality $h<m-1$ is strict, bearing in mind that for group $(iii)$ we have $(h,k)\neq (m-1,1)$ and $k\ge2$. Then \eqref{inequality-non-W} is greater than
\begin{align*}
& 2k\Big(m-h+j-\frac{1}{2}\Big)-h+1 \\
> & 2k\Big(\frac12+j\Big)-m+2\\
> & 2(1+j)-m+1\\
= & 2j+3-m=v_{\min}.
\end{align*}

\noindent Therefore we find that the sum has valuation $v_{\min}$, which is negative for sufficiently large $m$, thus 
\[
\nu_p\left(H(p^jn)\right)-m=\nu_p\left(\frac{H(p^jn)}{p^m}\right)=v_{\min}=2j+3-m,
\]
that means
\[
\nu_p\left(H(p^jn)\right)=2j+3.
\]

\item If instead there exists $m\in\N$ such that $\nu_p(H(p^mn)) \leq 2$, then by the descent phenomenon there exists $M\!\in\N$ such that $\nu_p(H(p^Mn))=0$. Then again evaluating formula \eqref{m_power:first} in $H(p^M\,p^jn)$ with $m=M$ in $p^jn$, where $p\nmid n$, we have 
\begin{equation*}
H(p^M\,p^jn)
=
\frac{H(p^jn)}{p^M}
+\sum\limits_{h=2}^{M-1}
\sum\limits_{k=1}^{h}
\frac{B_{p^h(p-1)-2k}}{2kp^{h}}\binom{p^{h}(p-1)-1}{2k-1}p^{2k(M-h+j)}n^{2k}+O(p),
\end{equation*}
where the LHS has valuation $-j$, and by the same reasoning as in case $(2)$ the RHS is a sum of two terms of valuation respectively $\nu_p(H(p^jn))-M$ and $2j+3-M$.
Now surely both valuation cannot exceed $-j$, since then the LHS and RHS would then have different valuations. Thus, we have the following cases:
\begin{itemize}
\item[•] If $3j+3>M$, then $\nu_p(H(p^jn))=M-j$.
\item[•] If $\nu_p(H(p^jn))>M-j$, then $3j+3=M$.
\item[•] If $3j+3\leq M$ and $\nu_p(H(p^jn))\leq M-j$, then $\nu_p(H(p^jn))=2j+3$.
\end{itemize}
Therefore, by rearranging the results, we conclude the proof.
\end{enumerate}

\vspace{-.5cm}

\end{proof}

\subsection{Proof of Theorem \ref{W-primes}}
Using the same arguments we can prove the analogous theorem for Wolstenholme primes.
\begin{proof}[Proof of Theorem \ref{W-primes}]
\noindent
\begin{enumerate}
\item  Again, using the same reasoning as in the previous Theorem, we evaluate formula \eqref{m_power:first} in $m=3$:
\[
H(p^3n)
=
 \frac{H(n)}{p^3}-B_{p^2(p-1)-2}\frac{n^2}{2}+O(p)
 =
 \frac{H(n)}{p^3}+\frac{B_{p-3}}{3}n^2+O(p)
 =
 \frac{H(n)}{p^3}+O(p),
\]
where the RHS has valuation zero. We note that this fact proves the descent phenomenon in the case $m=3$ for Wolstenholme primes.

\item As before, we evaluate formula \eqref{m_power:first} in $p^jn$, where $p\nmid n$:
\begin{align*}
H(p^{m+j}n)
&=
O(p)\quad\mbox{and}
\\
H(p^{m+j}n)
&=
\frac{H(p^jn)}{p^m}
+\sum\limits_{h=2}^{m-1}
\sum\limits_{k=1}^{h}
\frac{B_{p^h(p-1)-2k}}{2kp^{h}}\binom{p^{h}(p-1)-1}{2k-1}p^{2k(m-h+j)}n^{2k}+O(p).
\end{align*}

\noindent Differently from the proof of Theorem \ref{non-W-primes} the minimum valuation term in the sum is not known, but the lowest valuation of each summand is $v_{\min}\geq 2j+4-m$ according to the same argument used in the previous proof. Thus, since the LHS has positive valuation, we have
\[
\nu_p\left(\dfrac{H(p^jn)}{p^m}\right)\geq 2j+4-m,
\]
otherwise, thanks to ultrametric inequality, we would have
\[
\nu_p\left(H(p^m\,p^jn)\right)=2j+4-m
\]
that contradicts the hypothesis for sufficiently large $m$.

\noindent If, instead, there exists $m\in\N$ such that $\nu_p(H(p^mn))\leq3$ then by descendent phenomenon there exists $M\in\N$ such that $\nu_p(H(p^Mn))=0$. Therefore, we have:
\[
\nu_p(H(p^{M+j}n))=-j.
\]
Then again, evaluating formula \eqref{m_power:first} for $m=M$ in $p^jn$, where $p\nmid n$, we have the following congruence:

\begin{align*}
H(p^M\,p^jn)
= &
\frac{H(p^jn)}{p^M}
+\sum\limits_{h=2}^{M-1}
\sum\limits_{k=1}^{h}
\frac{B_{p^h(p-1)-2k}}{2kp^{h}}\binom{p^{h}(p-1)-1}{2k-1}p^{2k(M-h+j)}n^{2k}+O(p)\\
=&\dfrac{\Delta}{p^M}+\Sigma,
\end{align*}
say.
\noindent In this formula the $p$-adic valuation of the three terms are:
\begin{itemize}
    \item [$\bullet$] $\nu_p(H(p^M\,p^jn))=-j$,
    \item [$\bullet$] $\nu_p\left(\dfrac{\Delta}{p^M}\right)=\nu_p(\Delta)-M$,
    \item [$\bullet$] $\nu_p(\Sigma)\geq 2j+4-M$.
\end{itemize}

\noindent In order to complete the proof, we will use the $p$-adic valuation inequality to find a bound for $\nu_p(\Delta)$. However, in contrast to the situation in Theorem \ref{non-W-primes} case (3), there is not the exact valuation of $\nu_p(\Sigma)$. Hence, the possibilities are:
\begin{align*}
    i)&\quad\nu_p(\Delta)-M=-j &\nu_p(\Sigma)>-j\\
    ii)&\quad\nu_p(\Delta)-M>-j &\nu_p(\Sigma)=-j\\
    iii)&\quad\nu_p(\Delta)-M = \nu_p(\Sigma) \leq -j
\end{align*}

\noindent We combine these results with the inequality for $\nu_p(\Sigma)$ in order to estimate $\nu_p(\Delta)$. It is important to notice that the condition $\nu_p(\Sigma)>-j$ does not provide information combined with the other inequality concerning $\nu_p(\Sigma)$.
\begin{align*}
    i.a)&\quad\nu_p(\Delta)=M-j<2j+4 &\text{if }&M<3j+4 \text{ and $(i)$ holds,}\\
    i.b)&\quad\nu_p(\Delta)=M-j\geq 2j+4 &\text{if }&M\geq 3j+4 \text{ and $(i)$ holds,}\\
    ii)&\quad\nu_p(\Delta)>M-j>2j+4 &\text{if }&M \geq 3j+4 \text{ and $(ii)$ holds,}\\
    iii)&\quad\nu_p(\Delta)=\nu_p(\Sigma)+M\geq 2j+4 &\text{if }&M \geq 3j+4 \text{ and $(iii)$ holds.}
\end{align*}
This yields the desired result.
\end{enumerate}

\vspace{-.5cm}

\end{proof}

\section{Conclusions and further developments}
We highlight the fact that the existences of pairs $(n,p)$ satisfying case $(2)$ in Theorems \ref{non-W-primes} or \ref{W-primes} would disprove the conjecture of Eswarathasan and Levine on the finiteness of $J_p$: to achieve this outcome in formula \eqref{m_power:first} the sum should somehow vanish in valuation. We also point out that alas the refutation of case $(2)$ unfortunately does not imply Eswarathasan and Levine's conjecture, or even Boyd's conjecture on the impossibility of $\nu_p(H(n))\geq 4$, which would require the cases $(1),(2)$ to be false, and $M=3$ in case $(3)$.

One possible way to use the results achived could be that one undertaken by Sanna \cite{sanna} to ``count'' the number of elements of $J_p$ with $p$-adic valuation greater than 1. One could try to count the number of elements with $p$-adic valuation greater or equal to a positive integer $k$ in order to refute case 2 of Theorem \ref{non-W-primes}.
The problem in this case is that Sanna's Lemma 2.2 on short intervals works when looking for elements that cancel the derivative in a field $\Z_p$ ($f'_{d}\equiv 0\ \mod p$). As this case concerns a ring $\Z_{p^k}$ the number of roots of a polynomial cannot be bounded, in general, by $p$. Bounding it by $(d-1)p^{k-1}$ leads to overestimating the roots and then to a choice of $z$ that is inadequate for our purpose.

According to the same arguments used in this paper it does not appear too difficult to extend Theorems \ref{non-W-primes}- \ref{W-primes} for the generalized harmonic sum
\[
H_\al(n):=\sum\limits_{k=1}^n \frac{1}{k^\al}
\]
\noindent in the case $(p-1)\nmid \al$, the only difference being that the descent phenomenon lowers the valuation by $\al$.

\section{Acknowledgments}
This research is the result of a collaborative initiative among students of the three-year degree programme in Mathematics who have actively participated in the paths of excellence promoted by the Department of Mathematics ``Guido Castelnuovo'' of Sapienza University of Rome.

We would also like to thank Domenico Marino, Daniele Saracino and Francesco Vertucci for their fruitful suggestions that improved our paper.

We finally would like to thank the anonymous referees for their helpful comments and suggestions.

\end{document}